\documentclass[11pt]{amsart}
\usepackage{a4wide}
\usepackage[all]{xy}
\usepackage{amssymb,color}

\newcommand{\E}{\mathcal E}
\newcommand{\M}{\mathcal M}
\newcommand{\w}{\omega}
\newcommand{\A}{\mathcal A}
\newcommand{\IR}{\mathbb R}
\newcommand{\e}{\varepsilon}
\newcommand{\diam}{\mathrm{diam}}
\newcommand{\Homeo}{\mathcal H}
\newcommand{\IN}{\mathbb N}
\newcommand{\so}{\mathsf{so}}
\newcommand{\uc}{\mathsf{uc}}

\newcommand{\F}{\mathcal F}
\newcommand{\osc}{\mathrm{osc}}
\newcommand{\supp}{\mathrm{supp}}
\newcommand{\non}{\mathrm{non}}
\newcommand{\Ra}{\Rightarrow}
\newcommand{\as}{\mathsf \Delta^\circ_\w}

\newtheorem{theorem}{Theorem}[section]
\newtheorem{proposition}[theorem]{Proposition}
\newtheorem{corollary}[theorem]{Corollary}

\newtheorem{lemma}[theorem]{Lemma}
\newtheorem{problem}[theorem]{Problem}
\newtheorem{claim}{Claim}

\theoremstyle{definition}
\newtheorem{definition}[theorem]{Definition}
\newtheorem{example}[theorem]{Example}
\newtheorem{remark}[theorem]{Remark}

\title{Constructing a coarse space with a given Higson or binary corona}
\author{Taras Banakh and Igor Protasov}
\address{T.Banakh: Jan Kochanowski University in Kielce (Poland) and Ivan Franko National University of Lviv, Universytetska 1, 79000, Lviv, Ukraine}
\email{t.o.banakh@gmail.com}
\address{I.Protasov: Faculty of Computer Science and Cybernetics, Kyiv University,          Academic Glushkov pr. 4d, 03680 Kyiv, Ukraine}
\email{i.v.protasov@gmail.com}

\begin{document}
\begin{abstract} For any compact Hausdorff space $K$ we construct a canonical finitary coarse structure $\E_{X,K}$ on the set $X$ of isolated points of $K$. This construction has two properties:
\begin{enumerate}
\item If a finitary coarse space $(X,\E)$ is metrizable, then its coarse structure $\E$ 
coincides with the coarse structure $\E_{X,\bar X}$ generated by the Higson compactification $\bar X$ of $X$.
\item A compact Hausdorff space $K$ coincides with the Higson compactification of the coarse space $(X,\E_{X,K})$ if the set $X$ is dense in $K$ and the space $K$ is Fr\'echet-Urysohn.
\end{enumerate}
This implies that a compact Hausdorff space $K$ is homeomorphic to the Higson corona of some finitary coarse space  if one of the following conditions holds: (i) $K$ is perfectly normal; (ii) $K$ has weight $w(K)\le\w_1$ and character $\chi(K)<\mathfrak p$. Under CH every (zero-dimensional) compact Hausdorff space of weight $\le\w_1$ is homeomorphic to the Higson (resp. binary) corona of some cellular finitary coarse space.
\end{abstract}
\maketitle

\section{Introduction}

This paper was motivated by the following 

\begin{problem}\label{prob0} Which compact Hausdorff spaces are homeomorphic to the Higson compactifications or the Higson coronas of (finitary) coarse spaces?
\end{problem}

Answering this problem we shall prove the following statements: 
\begin{itemize}
\item each first-countable compact Hausdorff space with dense set of isolated points is homeomorphic to the Higson compactification of some finitary coarse space; 
\item each perfectly normal compact Hausdorff space is homeomorphic to the Higson corona of a suitable finitary coarse space;
\item under CH, each (zero-dimensional) compact Hausdorff space of weight $\le\w_1$ is homeomorphic to the Higson (resp. binary) corona of some cellular finitary coarse space.
\end{itemize}
Also we present simple examples of  compact Hausdorff spaces with countable dense set of isolated points, which are not homeomorphic to the Higson compactifications of finitary coarse spaces.

It should be mentioned that for compact metrizable spaces the ``corona'' part of Problem~\ref{prob0} was resolved by Mine and Yamashita \cite{MY} who proved that each compact metrizable space is homeomorphic to the Higson corona of a suitable locally compact metric space endowed with the $C_0$ coarse structure of Wright \cite{Wri}.

The paper is organized as follows. In Section~\ref{s:def} we recall the necessary information on coarse spaces and present (and characterize) two important classes of coarse spaces: metrizable and finitary. In Section~\ref{s:H}, given a proper metric space $M$ and a coarse space $(X,\E)$, we introduce two compact Hausdorff spaces $h_M(X,\E)$ and $\hbar_M(X,\E)$, called the $M$-compactification and the $M$-corona of the coarse space $(X,\E)$, respectively. The compact spaces $h_\IR(X,\E)$ and $\hbar_\IR(X,\E)$ (resp. $h_{\mathsf 2}(X,\E)$ and $\hbar_{\mathsf 2}(X,\E)$) coincide with the Higson compactification and the Higson corona (resp.  the binary compactification and the binary corona) of the coarse space $(X,\E)$, introduced by Higson \cite{Higson} (resp. Protasov \cite{Prot_bc}, \cite{Prot05}). In Section~\ref{s:P}, for any compact Hausdorff space $K$ we define  a finitary coarse structure $\E_{X,K}$ on the set $X$ of isolated points of $K$ and call the obtained coarse space $(X,\E_{X,K})$ {\em the permutation coarse space} of the compact space $K$. The main result of Section~\ref{s:M} is Theorem~\ref{t:M} saying each metrizable finitary coarse space $(X,\E)$ coincides with the permutation coarse space $(X,\E_{X,\bar X})$ of the Higson compactification $\bar X$ of $X$. In Section~\ref{s:soft} we define the notion of an $\IR$-soft compactification and prove that a compact Hausdorff space $K$ with dense set of isolated points is $\IR$-soft if and only if $K$ coincides with the Higson compactification of its permutation coarse space. Moreover, in Theorem~\ref{t:mainCH} we prove that under CH any $\IR$-soft separable compact space of weight $\le\w_1$ with dense set of isolated points is homeomorphic to the Higson compactification of some cellular finitary coarse space. In Section~\ref{s:12soft} we provide some sufficient conditions for $\IR$-softness, in particular, we observe that each Fr\'echet-Urysohn compact Hausdorff space with dense set of isolated points is $\IR$-soft. 
In Section~\ref{s:nonsoft} we present some examples of compactifications which are not soft. In final Section~\ref{s:Hcor} we address the problem of recognizing the Higson coronas of finitary coarse spaces and prove that a compact Hausdorff space $K$ is homeomorphic to the Higson corona of some finitary coarse space if  $K$ is perfectly normal (more generally, is a corona of character $\chi(K)<\mathfrak p$). Moreover, we prove that under CH, each (zero-dimensional) compact Hausdorff space of weight $\le\w_1$ is homeomorphic to the Higson (resp. binary) corona of some cellular finitary coarse space. 
\smallskip

\section{Coarse spaces}\label{s:def}

In this section we recall the necessary definitions from Asymptotic Topology (see \cite{PB}, \cite{PZ} and \cite{Roe} for more details). 

For a set $X$, any subset $E\subseteq X\times X$ containing the diagonal $\Delta_X:=\{(x,x):x\in X\}$ of $X\times X$ is called an {\em entourage} on $X$. For entourages $E,F$ on $X$ the sets
$$
\begin{aligned}
&E^{-1} :=\{(y,x):(x,y)\in E\}\mbox{ and }\\
&E\circ F:=\{(x,z):\exists y\in X\mbox{ with } (x,y)\in E\mbox{ and }(y,z)\in F\}
\end{aligned}
$$are entourages on $X$. An entourage $E$ on $X$ is called {\em cellular} if $E$ is an equivalence relation, i.e., $E=E^{-1}=E\circ E$.

For an entourage $E\subseteq X\times X$ and point $x\in X$ the set $E[x]:=\{y\in X:(x,y)\in E\}$ is called the {\em $E$-ball centered at} $x$.  For a subset $A\subseteq X$ the set $E[A]=\bigcup_{a\in A}E[a]$ is the {\em $E$-neighborhood} of $A$.

A {\em ballean} is a pair $(X,\E)$ consisting of a set $X$ and a family $\E$ of entourages on $X$ satisfying two axioms:
\begin{itemize}
\item $\bigcup\E=X\times X$.
\item for any sets $E,F\in\E$ there exists a set $G\in\E$ such that $E\circ F^{-1}\subseteq G$.
\end{itemize}
The family $\E$ is called the {\em ball structure} of the ballean $(X,\E)$. 

A ballean $(X,\E)$ is called a {\em coarse space} if for any sets $\Delta_X\subseteq E\subseteq F\subseteq X\times X$ the inclusion $F\in\E$ implies $E\in\E$. In this case the family $\mathcal E$ is called the {\em coarse structure} of the coarse space $(X,\E)$.

For a coarse space $(X,\E)$ a subfamily $\mathcal B\subseteq \mathcal E$ is called a {\em base} of $\E$ if each entourage $E\in\E$ is contained in some entourage $B\in\mathcal B$. It is easy to see that any base $\mathcal B$ of a coarse structure is a ball structure. Conversely, each ball structure $\mathcal B$ on a set $X$ is a base of a unique coarse structure
$${\downarrow}\mathcal B:=\{E:\exists B\in\mathcal B\;\;\Delta_X\subseteq E\subseteq B\}.$$

A coarse space is called {\em cellular} if its coarse structure has a base consisting of cellular entourages. More information on cellular coarse spaces can be found in \cite[\S3.1]{PZ}.

A subset $B\subseteq X$ of a coarse space $(X,\E)$ is called {\em bounded} if $B\subseteq E[x]$ for some $E\in\E$ and some $x\in X$. The family of all bounded subsets is called the {\em bornology} of the coarse space. 

Two subsets $A,B\subseteq X$ of a coarse space $(X,\E)$ are called {\em asymptotically separated} if for any entourage $E\in\E$ the intersection $E[A]\cap E[B]$ is bounded. 

Now we consider two important examples of coarse spaces.

\begin{example} Each metric space $(X,d)$ carries the canonical coarse structure $\E_d$ generated by the countable base consisting of the entourages
$$E_n:=\{(x,y)\in X\times X:d(x,y)\le n\}\mbox{ \ where $n\in\w$}.$$
\end{example}

\begin{definition} A coarse space $(X,\E)$ is called {\em metrizable} if $\E=\E_d$ for some metric $d$ on $X$. 
\end{definition}

The following characterization of metrizable coarse spaces can be found in \cite[2.1.1]{PZ}.

\begin{theorem} A coarse space $X$ is metrizable if and only if its coarse structure has a countable base.
\end{theorem}

A coarse space is called {\em ordinal} if its coarse structure has a well-ordered base.  
In particular, each metrizable coarse space is ordinal.
\smallskip

We recall that a {\em $G$-space} is a set $X$ endowed with an action of a group $G$ (which can be identified with a subgroup of the permutation group $S_X$ of $X$).

\begin{example}\label{ex2} Each $G$-space $X$ carries the canonical coarse structure $\E_G$ generated by the base consisting of the entourages
$$E_F:=\{(x,y)\in X\times X:y\in\{x\}\cup\{g(x)\}_{g\in F}\}$$where $F$ runs over finite subsets of the acting group $G$. 
\end{example}

Coarse spaces whose coarse structure is generated by the action of some group can be characterized as finitary coarse spaces.

\begin{definition} A coarse space $(X,\E)$ is called {\em finitary} if $$\sup\{|E[x]|:x\in X\}<\w$$ for all $E\in\E$.
\end{definition}

 It is clear that the bornology of any finitary coarse space $(X,\E)$ coincides with the family $[X]^{<\w}$ of all finite subsets of $X$. The following characterization of finitary coarse spaces was proved by Protasov in \cite{Prot} (see also \cite{PP} and \cite{P19}).

\begin{theorem}\label{t:Prot} A coarse space $(X,\E)$ is finitary if and only if $\E=\E_G$ for some subgroup $G\subseteq S_X$.
\end{theorem} 

\section{The $M$-compactification and $M$-corona of a coarse space}\label{s:H}

Let  $(M,d_M)$ be a proper metric space. The properness of $M$ means that $|M|\ge 2$ and each closed bounded set in $M$ is compact. For a non-empty subset $A\subseteq M$ let $$\diam(A):=\sup\{d_M(x,y):x,y\in A\}$$be the {\em diameter} of the set $A$ in $M$. 

A function $\varphi:X\to M$ on a coarse space $(X,\E)$ is called {\em slowly oscillating} if for every $\e>0$ and $E\in\E$ there exists a bounded set $B\subseteq X$ such that $\diam\, \varphi(E[x])<\e$ for all $x\in X\setminus B$. 

A coarse space $(X,\E)$ is called {\em $M$-normal} if for any disjoint and asymptotically separated nonempty sets $A,B\subseteq X$ there exists a slowly oscillating function $\varphi:X\to M$ such that $$\inf\{d_M(\varphi(a),\varphi(b)):a\in A,\;b\in B\}>0.$$
$\IR$-Normal coarse spaces are called {\em normal}. They were introduced by Protasov \cite{Prot03} and studied in \cite{BP}. By \cite{Prot03}, each metrizable (and cellular) coarse space is $\IR$-normal (and $\mathsf 2$-normal). Here by $\mathsf 2$ we denote the doubleton $\{0,1\}$ endowed with the (unique) $\{0,1\}$-valued metric.

Given a coarse space $(X,\E)$, denote by $\so(X,\E;M)$ the family of all bounded slowly oscillating $M$-valued functions on $X$. It can be shown that the function 
$$\delta_M:X\to M^{\so(X,\E;M)},\;\;\delta:x\mapsto (f(x))_{f\in\so(X,\E;M)},$$is injective. The properness of the metric space $M$ ensures that the closure $\overline{\delta_M(X)}$ of $\delta_M(X)$ in $M^{\so(X,\E;M)}$ is compact. 

The compact Hausdorff space $h_M(X,\E):=\overline{\delta_M(X)}$ is called the {\em $M$-compactification} of the coarse space $(X,\E)$. The injectivity of the map $\delta_M:X\to h_M(X,\E)$ allows us to identify $X$ with its image $\delta_M(X)$ in $h_M(X,\E)$. Observe that a function $f:X\to M$ is slowly oscillating if it is constant outside of some bounded set in $X$. This observation implies that  $X$ coincides with the set of isolated points of the compact space $h_M(X,\E)$, and for any bounded set $B\subseteq X$ its closure $\bar B$ in $h_M(X,\E)$ is open. Consequently, the set
$$\hbar_M(X,\E):=h_M(X,\E)\setminus\bigcup\{\bar B:B\subseteq X\mbox{ is bounded}\},$$
called the {\em $M$-corona} of $(X,\E)$, is compact. If all bounded sets in $X$ are finite, then the $M$-corona $\hbar_M(X,\E)$ coincides with the remainder $h_M(X,\E)\setminus X$ of the $M$-compactification of $X$.

The $\IR$-compactification $h_\IR(X,\E)$ and the $\IR$-corona $\hbar_\IR(X,\E)$ of a coarse space $(X,\E)$ were introduced by Higson \cite{Higson} and are called the {\em Higson compactification} and the {\em Higson corona} of $(X,\E)$, respectively. 

For the doubleton $\mathsf 2:=\{0,1\}$, the $\mathsf 2$-compactification $h_{\mathsf 2}(X,\E)$ and the $\mathsf 2$-corona $\hbar_{\mathsf 2}(X,\E)$ of a coarse space $(X,\E)$ were introduced by Protasov \cite{Prot_bc}, \cite{Prot05} and are called the {\em binary compactification} and the {\em binary corona} on $(X,\E)$, respectively.  Observe that the binary compactification (and the binary corona) of a coarse space is zero-dimensional. In Geometric Group Theory, the binary corona is known as the space of ends, see \cite{Yves}.
 
It is known that for any unbounded metrizable coarse space $(X,\E)$, the Higson corona $\hbar_\IR(X,\E)$ contains a copy of the  remainder $\beta\IN\setminus\IN$ of the Stone-Chech compactification of the discrete space $\IN$ of positive integers, so has rather complicated topological structure.

In the following sections we shall discuss the following 

\begin{problem}\label{prob0a} Which compact Hausdorff spaces are homeomorphic to the $M$-compactification (or the $M$-corona) of a finitary coarse space?
\end{problem}

\section{Compactifications and their permutation coarse spaces}\label{s:P}




Let $M$ be a proper metric space. 
In order to answer Problem~\ref{prob0a}, observe that for any coarse metric space $X$ its $M$-compactification is a compact Hausdorff space containing $X$ as a dense discrete subspace. 

Motivated by this observation, we define a compact Hausdorff space $K$ to be  {\em a compactification} if the set $K'$ of isolated points of $K$ is dense in $K$. So, $K$ is indeed a compactification of the discrete space $K'$. In particular, the $M$-compactification $h_M(X,\E)$ of a coarse space $(X,\E)$ is a compactification (of the discrete space $X$). 

For a compact Hausdorff space $K$ by $C(K;M)$ we denote the family of all continuous functions from $K$ to $M$. For a subset $X\subset K$, let
$$\uc(X,K;M):=\{f{\restriction}X:f\in C(K;M)\}.$$

We shall say that two compactifications $K_1,K_2$ of the same set $X=K'_1=K'_2$ are 
\begin{itemize}
\item {\em equivalent} if there exists a homeomorphism $h:K_1\to K_2$ such that $h(x)=x$ for any $x\in X$;
\item {\em $M$-equivalent} if $\uc(X,K_1;M)=\uc(X,K_2;M)$.
\end{itemize}

\begin{proposition} Let $K_1,K_2$ be two compactifications of the same set $X=K_1'=K_2'$.
\begin{enumerate}
\item If $K_1$ and $K_2$ are equivalent, then they are $M$-equivalent;
\item If $K_1$ and $K_2$ are homeomorphic to subspaces of $M^{\kappa}$ for some cardinal $\kappa$, then $K_1$ and $K_2$ are equivalent if and only if they are $M$-equivalent;
\item $K_1,K_2$ are equivalent if and only if they are $\IR$-equivalent;
\item If $K_1,K_2$ are zero-dimensional, then they are equivalent if and only if they are $\mathsf 2$-equivalent.
\end{enumerate}
\end{proposition}

\begin{proof} 1. If $K_1,K_2$ are equivalent, then there exists a homeomorphism $h:K_1\to K_2$ such that $h(x)=x$ for all $x\in X$. Given any function $\varphi\in\uc(X,K_2;M)$, find a continuous function $\bar \varphi:K_2\to M$ such that $\varphi=\bar\varphi{\restriction}X$. It follows that the function $\psi=\bar\varphi\circ h:K_1\to M$ is continuous and its restriction $\varphi=\psi{\restriction}X$ belongs to the family $\uc(X;K_1;M)$. This proves the inclusion $\uc(X,K_2;M)\subseteq  \uc(X,K_1;M)$. By analogy we can prove that $\uc(X,K_1;M)\subseteq \uc(X,K_2;M)$.
So the compactifications $K_1,K_2$ are $M$-equivalent.
\smallskip

2. Now assume that $K_1,K_2$ are $M$-equivalent and for every $i\in\{1,2\}$ the compact space $K_i$ admits a topological embedding $\delta_i:K_i\to M^{\kappa}$, where $\kappa$ is a suitable cardinal. We claim that the map $\delta_1{\restriction}X:X\to K_2\subset M^\kappa$ has a continuous extension $K_1\to K_2$. For every $\alpha\in\kappa$ consider the coordinate projection $\pi_\alpha:M^\kappa\to M$ onto the $\alpha$th coordinate. Observe that the restriction $\pi_\alpha{\restriction}K_2:K_2\to M$ is continuous and  the restriction $\pi_\alpha{\restriction}X\in \uc(X,K_2;M)=\uc(X,K_1;M)$ admits a unique continuous extension $\bar\pi_\alpha:K_1\to M$. Then the map $f=(\bar\pi_\alpha)_{\alpha\in\kappa}:K_1\to M^\kappa$ is a continuous extension of the map $\delta_1{\restriction}X$. The density of $X$ in the compact spaces $K_1,K_2$ ensures that $f(K_1)=K_2$. By analogy we can prove that the identity map $X\to K_1$ extends to a continuous map $g:K_1\to K_2$. Taking into account that $f\circ g(x)=g\circ f(x)=x$ for all $x\in X$, we conclude that $f:K_1\to K_2$ is a homeomorphism witnessing that the compactifications $K_1,K_2$ are equivalent. 
\smallskip

3,4. The last two statements follow immediately from the second statement.
\end{proof}

For a compact Hausdorff space $K$ let $X$ be the set of isolated points of $K$ and $\Homeo(K,X)$ be the group of homemorphisms $g$ of $K$ such that $g(x)=x$ for all $x\in K\setminus X$. The group $\Homeo(K,X)$ induces the subgroup
$$G_{X,K}:=\{g{\restriction}X:g\in\Homeo(K,X)\}$$
in the permutation group $S_{X}$ of the set $X$. The group $G_{X,K}$ will be called the {\em permutation group} of the compactification $K$.

The action of the group $G_{X,K}$ on the set $X$ determines the coarse structure $\E_{X,K}:=\E_{G_{X,K}}$ on the set $X$. The coarse space $(X,\E_{X,K})$ will be called the {\em permutation coarse space} of the compact Hausdorff space $K$.  

The construction of the permutation coarse space rises two problems, the first of which is related to Problems~\ref{prob0} and \ref{prob0a}.

\begin{problem}\label{prob1} Characterize compactifications  that are equivalent to the Higson compactifications of their permutation coarse spaces.
\end{problem}

\begin{problem}\label{prob2} Characterize coarse spaces that coincide with the permutation coarse spaces of their Higson compactifications.
\end{problem}

\section{The $M$-compactifications of finitary metrizable coarse spaces}\label{s:M}

A partial answer to Problem~\ref{prob2} is given by the following theorem.

\begin{theorem}\label{t:M} Let $M$ be a proper metric space, $(X,\E)$ be a coarse space,  $G:=G_{X,\bar X}$ be the permutation group of the $M$-compactification $\bar X:=h_M(X,\E)$ of $X$, and $\E_G$ be the coarse structure on $X$ generated by the action of the group $G$.\begin{enumerate}
\item If $(X,\E)$ is finitary, then $\E\subseteq\E_G$.
\item If $(X,\E)$ is metrizable and $M$-normal, then $\E_G\subseteq\E$.
\item If $(X,\E)$ is finitary, metrizable and $M$-normal, then $\E=\E_G$.
\end{enumerate}
\end{theorem}

\begin{proof} 1. Assuming that $(X,\E)$ is finitary, we shall prove that $\E\subseteq\E_G$. By Theorem~\ref{t:Prot}, $\E=\E_H$ for some subgroup $H\subseteq S_X$. We claim that $H\subseteq G_{X,\bar X}$. This will follow from the definition of the group $G_{X,\bar X}$ as soon as we prove that every $g\in H$ has a continuous extension $\bar g:\bar X\to\bar X$ such that $\bar g(x)=x$ for all $x\in \bar X\setminus X$. Since $\bar X$ embeds into the product $M^{\so(X,\E;M)}$ it suffices to prove that for every $\varphi\in\so(X,\E;M)$ the function $\psi:=\varphi\circ h:X\to M$ has a continuous extension $\bar\psi:\bar X\to M$ such that $\bar\psi(x)=x$ for all $x\in\bar X\setminus X$. Such extension $\bar\psi$ exists  if and only if for every $z\in \bar X\setminus X$ and $\e>0$ there exists a neighborhood $O_z\subseteq\bar X$ of $z$ such that $\diam\,\psi (X\cap O_x)<\e$. So, fix any point $z\in\bar X\setminus X$ and $\e>0$.

The permuation $g$ determines the entourage  $$E_g=\{(x,y)\in X\times X:y\in\{x,g(x)\}\}\in\E_H=\E.$$ Since the function $\varphi$ is slowly oscillating, there exists a bounded (and hence finite set) $B\subseteq X$ such that $\diam \,\varphi([E_g[x])<\frac14\e$ for all $x\in X\setminus B$.

By definition of $\bar X$, the slowly oscillating function $\varphi$ admits a continuous extension $\bar\varphi:\bar X\to M$. Then for the neighborhood $O_z:=\{x\in \bar X\setminus B:d_M(\bar\varphi(x),\bar\varphi(z))<\frac14\e\}$ of $z$ we have $\diam\, \psi(X\cap O_z)<\e$.  Indeed, for any $x,y\in X\cap O_z$, we get
$$d_M(\psi(x),\bar\varphi(z))\le d_M(\varphi\circ h(x),\varphi(x))+d_M(\varphi(x),\bar\varphi(z))<\tfrac14\e+\tfrac14\e=\tfrac12\e$$
and then 
$$d_M(\psi(x),\psi(y))\le d_M(\psi(x),\bar\varphi(z))+d_M(\psi(y),\bar\varphi(z))<\tfrac12\e+\tfrac12\e=\e.$$
This completes the proof of the inclusion $H\subseteq G_{X,\bar X}$. The latter inclusion implies $\E=\E_H\subseteq \E_{G_{X,\bar X}}$.
 \smallskip
 
2. Assume that the coarse space $(X,\E)$ is metrizable and $M$-normal. By Theorem~\ref{t:M}, the coarse structure $\E$ has a countable base $\{E_n\}_{n\in\w}\subseteq\E$ such that $E_n\subseteq E_{n+1}$ for all $n\in\w$. Assuming that $\E_G\not\subseteq\E$, we can find a finite subset $F\subseteq \Homeo(\bar X,X)$ such that the entourage $$E_F:=\{(x,y)\in X\times X:y\in\{x\}\cup\{g(x)\}_{g\in F}\}$$ does not belong to the coarse structure $\E$.
In particular, $E_F\not\subseteq E_n^{-1}\circ E_n$ for every $n\in\w$. Replacing $F$ by a larger finite set in $\Homeo(\bar X,X)$, we can assume that $F$ contains the identity homeomorphism of the space $\bar X$.

Take any point $x_0\in X$ and inductively construct a sequence of points $(x_n)_{n\in\w}$ in $X$ and bijections $g_n\in F$ such that for every $n\in\w$ 
$$g_n(x_n)\notin E_n^{-1}\circ E_n[x_{n}]\mbox{ and }x_{n}\notin \bigcup_{k<n}\bigcup_{g\in F}E_n^{-1}\circ E_n[g(x_k)].$$ 
We claim that the sets $A:=\{x_n\}_{n\in\w}$ and $B:=\{g_n(x_n)\}_{n\in\w}$ are asymptotically separated. Since $\{E_n\}_{n\in\w}$ is a base of the coarse structure $\E$, for every entourage $E\in\E$ there exists $n\in\w$ such that $E\subseteq E_n$. The choice of the points $(x_k)_{k\in\w}$ guarantees that $E[A]\cap E[B]\subseteq  E_n[A]\cap E_n[B]\subseteq \bigcup_{i\le n}E_n[x_i]$ is finite.

By the $M$-normality of the coarse space $(X,\E)$, there exists a slowly oscillating function $\varphi:X\to M$ such that $$\e:=\inf\{d_X(\varphi(a),\varphi(b)):a\in A,\;b\in B\}>0.$$ By the definition of the $M$-compactification, the slowly oscillating function $\varphi$ admits a continuous extension $\bar\varphi:\bar X\to M$. 

Since the set $F\subseteq \Homeo(\bar X,X)$ is finite, by the continuity, for every point $z\in \bar X\setminus X$ we can find a neighborhood $O_z\subseteq\bar X$ such that $$d_M(\bar\varphi(g(x)),\bar\varphi(g(z)))<\tfrac12\e$$ for any $g\in F$ and $x\in O_z$. In particular, $d_M(\bar\varphi(x),\bar\varphi(z))<\frac12\e$ for any $x\in O_z$. 

Next, choose an open neighborhood $U_z\subseteq O_z$ such that $g(U_z)\subseteq O_z$ for all $g\in F$. It follows that $U:=\bigcup_{z\in\bar X\setminus X}U_z$ is an open set containing $\bar X\setminus X$ and then $F:=\bar X_\E\setminus U$ is a compact (and hence finite) set in the discrete space $X$. 

Choose any $n\in\w$ with $x_n\notin F$ and find $z\in \bar X_\E\setminus X$ with $x_n\in U_z\subseteq O_z$. Then $g_n(x_n)\in O_z$ and  $$
\begin{aligned}
d_M(\varphi(g_n(x_n)),\varphi(x_n))&\le d_M(\varphi(g_n(x_n)),\bar\varphi(g_n(z)))+d_M(\bar\varphi(g_n(z)),\bar\varphi(z))+d_M(\bar\varphi(z),\varphi(x_n))\\
&<\frac\e2+0+\frac\e2=\e,
\end{aligned}
$$ which contradicts the definition of $\e$. This contradiction completes the proof of the inclusion $\E_G\subseteq\E$.
\smallskip

3. The last statement follows from the first two statements.
\end{proof}

Since each metrizable coarse space is $\IR$-normal \cite{Prot03}, Theorem~\ref{t:M} implies

\begin{corollary}\label{c:R} Let $(X,\E)$ be a coarse space,  $G:=G_{X,\bar X}$ be the permutation group of its Higson compactification $\bar X:=h_\IR(X,\E)$, and $\E_G$ be the coarse structure on $X$ generated by the action of the group $G$.\begin{enumerate}
\item If $(X,\E)$ is finitary, then $\E\subseteq\E_G$.
\item If $(X,\E)$ is metrizable, then $\E_G\subseteq\E$.
\item If $(X,\E)$ is finitary and metrizable, then $\E=\E_G$.
\end{enumerate}
\end{corollary}

\section{$M$-Soft compactifications}\label{s:soft}

Let $(M,d_M)$ be a proper metric space. 
 In order to answer Problems~\ref{prob0} and \ref{prob1}, we shall introduce the notion of an $M$-soft compactification. First we recall the definition of oscillation. Let $\varphi:X\to M$ be a function defined on a subset $X$ of a topological space $K$. For a point $z\in\bar X\subseteq K$ denote by $\tau_z$ the family of all open neighborhoods of $z$ in $K$. The number
$$\osc_z(\varphi):=\inf_{O_x\in\tau_x}\diam\,\varphi(X\cap O_x)$$
is called {\em the oscillation} of $\varphi$ at the point $z$. By Lemma 4.3.16 in \cite{Eng}, $\varphi$ admits a continuous extension $\bar X\to M$ if and only if $\osc_z(\varphi)=0$ for each $x\in\bar X$.

\begin{definition} A compactifiction $K$ is defined to be {\em $M$-soft} if for any bounded function $\varphi:K'\to M$ with $\{z\in K:\osc_z(\varphi)>0\}\ne\emptyset$, there exist a homeomorphism $g\in \Homeo(K,K')$ and $\e>0$ such that $d_M(\varphi(g(x)),\varphi(x))\ge\e$ for infinitely many points $x\in K'$.
\end{definition}

\begin{lemma}\label{l:soft} The $M$-compactification $h_M(X,\E)$ of any finitary coarse space $(X,\E)$ is $M$-soft.
\end{lemma}

\begin{proof} Let $(X,\E)$ be a finitary coarse space. To show that its $M$-compactification $K:=h_M(X,\E)$ is $M$-soft, fix any bounded function $\varphi:K'\to M$ with $\osc_z(\varphi)>0$ for some $z\in K$. This means that $\varphi$ does not  extend to a continuous function on $K$ and hence $\varphi\notin \so(X,\E;M)$. So, there exists an entourage $E\in\E$ and $\e>0$ such that the set $\Omega:=\{x\in X:\diam\,\varphi(E[x])\ge\e\}$ is unbounded and hence infinite in $X$.

By Theorem~\ref{t:Prot}, $\E=\E_G$ for some subgroup $G\subseteq S_X$. The definition of the coarse structure $\E_G$ yields a finite subset $F=F^{-1}\subseteq G$ such that $E\subseteq E_F:=\{(x,y):y\in\{x\}\cup \{g(x)\}_{g\in F}\}$. Then for any $x\in\Omega$ we have
$$\e\le \diam\, 
\varphi(E_F[x])\le 2\max_{g\in F}d_M(\varphi{\circ}g(x),\varphi(x)).$$
For every $g\in F$ consider the set $\Omega_g:=\{x\in\Omega:d_M(\varphi\circ g(x),\varphi(x))\ge\frac12\e\}$ and observe that $\Omega=\bigcup_{g\in F}\Omega_g$. 
By the Pigeonhole Principle, for some $g\in F$ the set $\Omega_g$ is infinite. So, $d_M(\varphi\circ g(x),\varphi(x))\ge\frac12\e$ for infinitely many points $x$ (in particular, for all points in $\Omega_g$).
\end{proof}

The following theorem answers Problem~\ref{prob1}.

\begin{theorem}\label{t:main2} Let $M$ be a proper metric space, $K$ be a compactification, $X$ be the set of isolated points of $K$ and $(X,\E_{X,K})$ be the permutation coarse space of $K$. The space $K$ is $M$-equivalent to the $M$-compactification $h_M(X,\E)$ of the coarse space $(X,\E_{X,K})$ if and only if $K$ is $M$-soft.
\end{theorem}

\begin{proof} The ``only if'' part follows from Lemma~\ref{l:soft}. To prove the ``if'' part, assume that the compactification $K$ is $M$-soft. Let $X$ be the set of isolated points of $K$. Denote by $C(K;M)$ the family of all continuous $M$-valued functions on $K$ and let $\uc(X,K;M):=\{f{\restriction}X:f\in C(K;M)\}$. The $M$-equivalence of the compactifications $K$ and $\bar X:=h_M(X,\E)$ will follow as soon as we check that $\uc(X,K;M)=\so(X,\E_{X,K};M)$.

First we check that each function $\varphi\in\uc(X,K;M)$ belongs $\so(X,\E_{X,K};M)$. Given any $\e>0$ and a finite subset $F\subseteq G_{X,K}$, we will find a finite set $B\subseteq X$ such that
 $\diam\{\varphi(g(x)):g\in F\}<\e$ for all $x\in X\setminus B$. 
 
Let $\bar\varphi:K\to M$ be a unique continuous function such that $\varphi=\bar\varphi{\restriction}X$. By the uniform continuity of $\bar\varphi$, there exists an open neighborhood $U\subseteq K\times K$ of the diagonal $\Delta_K$ such that $\diam\, \bar\varphi(U[z])<\frac12\e$ for any $z\in K$. By the definition of the group $G_{X,K}$, every permutation $g\in F\subseteq G_{X,K}$ coincides with the restriction $\bar g{\restriction}X$ of some homeomorphism $\bar g$ of $K$ such that $\bar g(z)=z$ for all $z\in K\setminus X$. Then any $z\in K\setminus X$ has an open neighborhood $V_z\subseteq K$ such that $V_z\times \bar g(V_z)\subseteq U$ for all $g\in F$. It follows that $V:=\bigcup_{z\in K\setminus X} V_z$ is an open neighborhood of $K\setminus X$ in $K$. Since the subspace $X$ is discrete, the compact set $B:=K\setminus V\subseteq X$ is finite. 

For any point $x\in X\setminus B=V$ there exits a point $z\in K\setminus X$ such $x\in V_z$. Then for any $g\in F$ we have $(x,\bar g(x))\in V_z\times \bar g(V_z)\subseteq U$ and hence $d_M(\varphi\circ \bar g(x),\varphi(x))<\frac12\e$. Finally, $\diam \{\varphi\circ g(x):g\in F\}<\frac12\e+\frac12\e=\e$ for any $x\in X\setminus B$. 
 \smallskip
 
Next, we show that each function $\varphi\in\so(X,\E_{X,K};M)$ belongs to $\uc(X,K;M)$. To derive a contradiction, assume that $\varphi$ does not belong to $\uc(X,K;M)$ and hence admits no continuous extensions to $K$.
This means that $\osc_z(\varphi)>0$ for some $z\in K\setminus X$.
Since $K$ is $M$-soft, there exist a permutation $g\in \Homeo(K,X)$ and $\e>0$ such that $d_M(\varphi(g(x)),\varphi(x))\ge\e$ for infinitely many points $x\in X$. This means that the function $\varphi$ is not slowly oscillating (with respect to the coarse structure $\E_{X,K}=\E_{G_{X,K}}$, which contradicts the choice of $\varphi$.
\end{proof}

\begin{corollary}\label{c:main}  A compactification $K$ is homeomorphic to the Higson compactification of some finitary coarse space if and only if $K$ is $\IR$-soft.
\end{corollary}

Under some additional set-theoretic assumptions, we can prove more. Namely, that every separable $\IR$-soft compactification is homeomorphic to the Higson compactification of some {\em cellular} finitary coarse space. This holds under the equality $\as=\mathfrak c$, where $\as$ is the smallest cardinality of a base of a cellular finitary coarse structure $\E$ on $\w$ containing no infinite $\E$-discrete subsets.  A subset $D$ of a coarse space $(X,\E)$ is called {\em $\E$-discrete} if for any entourage $E\in\E$ there exists a bounded set $B\subset X$ such that $D\cap E[x]=\{x\}$ for any $x\in D\setminus B$. In \cite{LP} $\E$-discrete sets are called thin.

The cardinal $\as$ has been introduced in \cite{j}, where it is shown that $\max\{\mathfrak b,\mathfrak s,\mathrm{cov}(\mathcal N)\}\le\as\le\mathfrak c$ and hence $\as=\mathfrak c$ under Martin's Axiom. On the other hand, we do not know whether $\as=\mathfrak c$ in ZFC. The equality $\as=\mathfrak c$ allows us to prove the following consistent improvement of the ``if'' part in Corollary~\ref{c:main}.

\begin{theorem}\label{t:mainCH} Let $M$ be a proper metric space. Under $\as=\mathfrak c$, each separable $M$-soft compactification $K$ is $M$-equivalent to the $M$-Higson compactification $h_M(K',\E)$ of the set $K'$ of isolated points of $K$ endowed with a suitable cellular finitary coarse structure $\E$.
\end{theorem}

\begin{proof} By the separability of $K$, the set $X$ of isolated point of $K$ is countable. Being proper, the metric space $M$ has cardinality $|M|\le\mathfrak c$. Let $C(K;M)$ be the family of continuous $M$-valued maps on $K$ and $\uc(X,K;M):=\{f{\restriction}X:f\in C(K;M)\}$. Let $\mathsf b(X;M)$ be the set of bounded $M$-valued functions on $X$. If $\uc(X,K;M)=\mathsf b(X;M)$, then each bounded function $\varphi:X\to M$ has a continuous extension $\bar\varphi:K\to M$, which implies that $K$ is a Stone-\v Cech compactification of $X$ and $K$ is $M$-equivalent to $h_M(X,\E_H)$ for the cellular finitary coarse structure $\E_H$,  induced by the group $H=FS_X$ of finitely supported permutations of $X$.

So, we assume that $\uc(X,K;M)\ne\mathsf b(X,K;M)$.  Since $|M^X|\le|\mathfrak c^\w|=\mathfrak c$, the non-empty set $\mathsf b(X;M)\setminus\uc(X,K;M)$ can be written as $\{\varphi_\alpha\}_{\alpha<\mathfrak c}$.

Let $G_{X,K}$ be the permutation group of the compactification $K$. A subgroup $H\subset G_{X,K}$ is defined to be {\em cellular} if the coarse structure $\E_H$ on $X$ is cellular. This happens if and only if for any finitely-generated subgroup $F\subseteq H$ the cardinal $\sup_{x\in X}|F(x)|$ is finite. Here $F(x):=\{f(x):f\in F\}$.

Assume that $\as=\mathfrak c$. We shall inductively construct an increasing transfinite sequence of cellular subgroups $(H_\alpha)_{\alpha\in\mathfrak c}$ of $G_{X,K}$ such that for any $\alpha<\mathfrak c$ the group $H_\alpha$ has cardinality $|H_\alpha|\le |\alpha+\w|$ and there are $\e_\alpha>0$ and $g_\alpha\in H_\alpha$ such that the set $\{x\in X:d_M(\varphi\circ g_\alpha(x),\varphi(x))>\e_\alpha\}$ is infinite.

To start the inductive construction, consider the function $\varphi_0$. By the $M$-softness of $K$, for the function $\varphi_0\notin \uc(X,K;M)$ there exist $\e_0>0$ and a permutation $h_0\in G_{X,K}$ such that the set $\Omega_0=\{x\in K':d_M(\varphi_0(h_0(x)),\varphi_0(x))>\e_0\}$ is infinite. Choose a sequence of points $x_n\in\Omega$ such that $x_{n+1}\notin \{x_k\}_{k\le n}\cup \{h_0(x_k)\}_{k\le n}\cup\{h_0^{-1}(x_k)\}_{k\le n}$ for all $n\in\w$. Consider the permutation $g_0:X\to X$, defined by
$$
g_0(x)=
\begin{cases}h_0(x_n)&\mbox{if $x=x_n$ for some $n\in\w$};\\
x_n&\mbox{if $x=h_0(x_n)$ for some $n\in\w$};\\
x&\mbox{otherwise};
\end{cases}
$$and observe that $g_0$ belongs to the group $G_{X,K}$. Let $H_0$ be the cyclic subgroup of $G_{X,K}$ generated by $g_0$. Observe that $|H_0|=2$ and hence the subgroup $H_0\subset S_X$ is cellular.

Now assume that for some nonzero ordinal $\alpha<\mathfrak c$ we have constructed an increasing transfinite sequence $(H_\beta)_{\beta<\alpha}$ of cellular subgroups of $S_{X,K}$ such that $|H_\beta|\le|\beta+\w|$ for all $\beta<\alpha$. Then the union $H_{<\alpha}=\bigcup_{\beta<\alpha}H_\beta$ is a cellular subgroup of $S_X$ with $|H_{<\alpha}|\le\sum_{\beta<\alpha}|H_\beta|\le |\alpha|\cdot |\alpha+\w|=|\alpha+\w|$ and $\E_{H_{<\alpha}}$ is a cellular finitary coarse structure on $\w$ having a base of cardinality $|[H_{<\alpha}]^{<\w}|\le|\alpha+\w|<\mathfrak c=\as$. Consider the function $\varphi_\alpha\in \mathsf b(X;M)\setminus \uc(X,K;M)$. By the $M$-softness of $K$, there exist a permutation $h_\alpha\in G_{X,K}$ and $\e_\alpha>0$ such that the set $\Omega_\alpha=\{x\in X:d_M(\varphi_\alpha\circ h_\alpha(x),\varphi_\alpha(x))>\e_\alpha\}$ is infinite.

Two cases are possible.

1. For some finitely-generated subgroup $F\subseteq H_{<\alpha}$, the set $\{x\in\Omega_\alpha:h_\alpha(x)\in F(x)\}$ is infinite.
 Let $S=S^{-1}$ be a finite set of generators of the group $F$. Let $e$ be the unit of the group $F$, $S^0=\{e\}$ and $S^{k+1}=S^k\circ S=\{f\circ g:f\in S^k,\;g\in S\}$ for $k\in\w$. Since the subgroup $F\subseteq H_{<\alpha}\subset S_X$ is cellular, the cardinal $m=\sup_{x\in X}|F(x)|$ is finite. We claim that for some $k<m$ the set  $\{x\in \Omega_\alpha:S^k(x)=S^{k+1}(x)\}$ is infinite. Assuming the converse, we conclude that the set $\Omega_\alpha'=\bigcup_{k< m}\{x\in \Omega_\alpha:S^k(x)=S^{k+1}(x)\}$ is finite and then for any $x\in\Omega_\alpha\setminus\Omega_\alpha'$ the sequence $(S^k(x))_{k=0}^m$ is strictly increasing which is not possible as $\bigcup_{k\in\w}S^k(x)=F(x)$ has cardinality $\le m$. This contradiction shows that for some $k\in\w$ the set $$\widetilde\Omega_\alpha:=\{x\in\Omega_\alpha:S^k(x)=S^{k+1}(x)\}=\{x\in\Omega_\alpha:S^k(x)=F(x)\}$$ is infinite. Then for every $x\in \widetilde\Omega_{\alpha}$ we have $h_\alpha(x)\in F(x)=S^k(x)$. By the Pigeonhole Principle,  for some permutation $s_\alpha\in S^k$, the set $\Lambda_\alpha=\{x\in \widetilde \Omega_{\alpha}:h_\alpha(x)=s_\alpha(x)\}$ is infinite. 
 
Put $H_\alpha:=H_{<\alpha}$ and $g_\alpha:=s_\alpha$, and observe that the subgroup $H_\alpha$ of $S_X$ is cellular and the set
$\{x\in X:d_M(\varphi_\alpha\circ g_\alpha(x),\varphi(x))>\e_\alpha\}$ contains the set $\Lambda_\alpha$ and hence is infinite.
\smallskip

2. For every finitely generated subgroup $F\subseteq H_{<\alpha}$ the set $\{x\in\Omega_\alpha:h_\alpha(x)\in F(x)\}$ is finite. Since the cellular finitary coarse structure $\E_{H_{<\alpha}}$ on $X$ has base of cardinality $<\mathfrak c=\as$, the definition of the cardinal $\as$ ensures that the infinite set $\Omega_\alpha$ contains an infinite $\E_{H_{<\alpha}}$-discrete subset $T_\alpha\subseteq\Omega_\alpha$. Find a sequence $\{x_n\}_{n\in\w}$ of pairwise distinct points of $T_\alpha$ such that the sets $\{x_n\}_{n\in\w}$ and $\{h_\alpha(x_n)\}_{n\in\w}$ are disjoint.

 Consider the permutation $g_\alpha:X\to X$, defined by
$$
g_\alpha(x)=
\begin{cases}h_\alpha(x_n)&\mbox{if $x=x_n$ for some $n\in\w$};\\
x_n&\mbox{if $x=h_\alpha(x_n)$ for some $n\in\w$};\\
x&\mbox{otherwise;}
\end{cases}
$$and observe that $g_\alpha$ belongs to the group $G_{X,K}$. Let $H_\alpha$ be the subgroup of $G_{X,K}$ generated by the set $\{g_\alpha\}\cup H_{<\alpha}$. Using Lemma~5.3 from \cite{j}, one can show that the subgroup $H_\alpha$ of $S_X$ is cellular and 
the set $\{x\in X:d_M(\varphi_\alpha\circ g_\alpha(x)),\varphi_\alpha(x))>\e_\alpha\}\supseteq\{x_n\}_{n\in\w}$ is infinite. This completes the inductive step.

After completing the inductive construction, consider the cellular subgroup $H=\bigcup_{\alpha<\mathfrak c}H_\alpha$ of $G_{X,K}$.
The property of the transfinite sequence $(H_\alpha)_{\alpha<\mathfrak c}$ ensures that any function $\varphi\in \mathbf b(X;M)\setminus \uc(X,K;M)$ is not slowly oscillating with respect to the coarse structure $\E_H$. Consequently, $\uc(K,h_M(X,\E_H);M)=\so(X,\E_H;M)\subseteq\uc(X,K;M)$. On the other hand,  Theorem~\ref{t:main2} ensures that
$\uc(X,K;M)=\so(X,\E_{G_{X,K}};M)\subseteq \so(X,\E_H;M)$, where the latter inclusion folows from $H\subseteq G_{X,K}$. Therefore, $\uc(X,K;M)=\uc(X,h_M(X,\E_H);M)$ and the compactifications $K$ and $h_M(X,\E_H)$ are $M$-equivalent. 
\end{proof}

Applying Theorem~\ref{t:mainCH} to the proper metric spaces $\IR$ and $\mathsf 2$, we obtain the following corollaries.

\begin{corollary}\label{c:CHR} Under $\as=\mathfrak c$, each separable $\IR$-soft compactification $K$ is homeomorphic to the Higson compactification of some cellular finitary coarse space.\end{corollary}

\begin{corollary}\label{c:CH0} Under $\as=\mathfrak c$, each zero-dimensional separable $\mathsf 2$-soft compactification $K$ is homeomorphic to the binary compactification of some cellular finitary coarse space.
\end{corollary}

\section{Soft compactifications}\label{s:12soft}

The results of the preceding section motivate a deeper study of $M$-soft compactifications. For this purpose, we introduce a stronger version of the $M$-softness.

\begin{definition} A compactification $K$  is defined to be  {\em soft} if for any disjoint sets $A,B\subseteq K'$ with $\bar A\cap\bar B\ne\emptyset$ there exists a homeomorpism $h\in\Homeo(K,K')$ such that the set $\{x\in A:h(x)\in B\}$ is infinite.
\end{definition}

\begin{proposition}\label{p:sss} Let $M$ be a proper metric space and $K$ be a compactification.
\begin{enumerate}
\item If $K$ is soft, then $K$ is $\IR$-soft.
\item If $K$ is $\IR$-soft, then $K$ is $M$-soft.
\item If $K$ is $M$-soft, then $K$ is $\mathsf 2$-soft.
\end{enumerate}
\end{proposition}

\begin{proof} 1. Assume that a compactification $K$ is soft. To show that $K$ is $\IR$-soft, fix any bounded function $\varphi:K'\to\IR$ with $\osc_z\varphi>0$ for some $z\in  K\setminus K'$. Let $\tau_z$ be the family of open neighborhoods of $z$ in $K$. By definition of the oscillation $\e:=\osc_z\varphi$, for every $U\in\tau_z$ the closed bounded set $\overline{\varphi(U\cap K')}$ in $\IR$ has diameter $\ge\e$. Then the intersection $I:=\bigcap_{U\in\tau_z}\overline{\varphi(U\cap K')}$ also has diameter $\ge \e$ and we can find two points $a,b\in I$ with $b-a=\e$. It follows that the sets
$$
A:=\{x\in X:\varphi(x)<\tfrac23a+\tfrac13b\}\mbox{ \ and \ }B:=\{x\in X:\varphi(x)>\tfrac13a+\tfrac23b\}
$$contain the point $z$ in their closures. Since the compactification $K$ is soft, there exists a homeomorphism $g\in\Homeo(K,K')$ such that the set $\Omega:=\{x\in A:g(x)\in B\}$ is infinite. Then for any $x\in\Omega$ we have
$$|\varphi(g(x))-\varphi(x)|\ge\varphi(g(x))-\varphi(x)>\tfrac13a+\tfrac23b-(\tfrac23a+\tfrac13b)=\tfrac13(b-a)=\tfrac13\e>0,$$
witnessing that the compactification $K$ is $\IR$-soft.
\smallskip

2. Assume that a compactification $K$ is $\IR$-soft. To prove that it is $M$-soft, fix any bounded function $\varphi:K'\to M$ such that $\e:=\osc_z(\varphi)>0$ for some $z\in K$. Let $\tau_z$ be the family of neighborhoods of $z$ in $K$. It follows that for any $U\in\tau_z$ the image $\varphi(U\cap K')$ is a bounded set of diameter $\ge\e$ in $M$. The properness of $M$ ensures that the intersection $I=\bigcap_{U\in\tau_z}\overline{\varphi(U\cap K')}$ is a compact set of diameter $\ge\e$ in $M$. By the compactness of $I$, there are points $a,b\in I$ with $d_M(a,b)\ge\e$. Observe that $f:M\to\IR$, $f:x\mapsto d_M(a,x)$, is a non-expanding function with  $|f(a)-f(b)|=d_M(a,b)\ge\e$. It follows that the function $f\circ\varphi:K'\to \IR$ has $\osc_z(f\circ\varphi)\ge\e>0$. By the $\IR$-softness of $K$, there exists a permutation $g\in G_{X,K}$ and a positive real number $\delta$ such that the set $\{x\in K':|f\circ\varphi(g(x))-f\circ\varphi(x)|>\delta\}$ is infinite. Since the function $f$ is non-expanding, the set $\{x\in K':d_M(\varphi(g(x)),\varphi(x))>\delta\}\supseteq  \{x\in K':|f\circ\varphi(g(x))-f\circ\varphi(x)|>\delta\}$ is infinite, too.
\smallskip

3. Assume that a compactification $K$ is $M$-soft. To show that $K$ is $\mathsf 2$-soft, fix any function $\varphi:K'\to \mathsf 2$ with $\osc_z(\varphi)>0$ for some $z\in K$. By definition, the proper space $M$ contains more than one point and hene admts an injective map $f:\mathsf 2\to M$. Then the map $f\circ\varphi:K'\to M$ has $\osc_z(f\circ\varphi)>0$ and by the $M$-softness of $K$, there exists a permutation $g\in G_{X,K}$ and $\e>0$ such that the set $\{x\in K':d_M(f\circ\varphi\circ g(x),f\circ \varphi(x))>\delta\}$ is infinite. Taking into account that for two points $x,y\in \mathsf 2$ the inequality $d_M(f(x),f(y))>0$ is equivalent to $|x-y|>\frac12$, we conlcude that the set   $\{x\in K':|\varphi\circ g(x)-\varphi(x)|>\frac12\}$ contains $\{x\in K':d_M(f\circ\varphi\circ g(x),f\circ \varphi(x))>\delta\}$ and hence is infinite.
\end{proof}

Proposition~\ref{p:sss} implies that for any proper metric space $M$ and any compactification we have the implications
$$\mbox{soft $\Ra$ $\IR$-soft $\Ra$ $M$-soft $\Ra$ $\mathsf 2$-soft}.$$

\begin{example} The Stone-\v Cech compactification $\beta\w$ of the countable discrete space $\w$ is soft.
\end{example}

\begin{proof} The softness of $\beta\w$ follows from the well-known fact \cite[3.6.2]{Eng} that disjoint sets $A,B\subseteq\w$ have disjoint closures in $\beta\w$.
\end{proof}

\begin{proposition}\label{p:ss} A compactification $K$ is soft if for any disjoint sets $A,B\subseteq K'$ with $\bar A\cap\bar B\ne\emptyset$ there are sequences $\{a_n\}_{n\in\w}\subseteq A$ and $\{b_n\}_{n\in\w}\subseteq B$ that converge to some point $z\in K$.\end{proposition}

\begin{proof} Given disjoint sets $A,B\subseteq K'$ with $\bar A\cap\bar B\ne\emptyset$ find sequences $\{a_n\}_{n\in\w}\subseteq A$ and $\{b_n\}_{n\in\w}\subseteq B$ that converge to some point $z\in K$. Since the sets $A,B$ are disjoint, the point $z$ is not isolated. So, we can assume that $a_n\ne a_m$ and $b_n\ne b_m$ for all $n\ne m$. It can be shown that the bijective function $g:K\to K$ defined by
$$g(x)=\begin{cases}
b_n&\mbox{if $x=a_n$};\\
a_n&\mbox{if $x=b_n$};\\
x&\mbox{otherwise};
\end{cases}
$$ is a homeomorphism of $K$ and $g\in\Homeo(K,K')$. For this homeomorphism the set $\{x\in A:g(x)\in B\}$ contains the set $\{a_n\}_{n\in\w}$ and hence is infinite.
\end{proof}

We recall that a topological space $X$ is {\em Fr\'echet-Urysohn} if for any set $A\subseteq X$ and point $x\in \bar A\setminus A$ there exists a sequence $\{a_n\}_{n\in\w}$ that converges to $x$. Proposition~\ref{p:ss} implies:

\begin{corollary}\label{c:fu} A compactification $K$ is soft if  the space $K$ is Fr\'echet-Urysohn.
\end{corollary}

Since each first-countable space is Fr\'echet-Urysohn, Corollary~\ref{c:fu} implies another corollary.

\begin{corollary} Each first-countable compactification  is soft.
\end{corollary}

By $\mathfrak p$ we denote the smallest cardinality $|\F|$ of a family $\F$ of infinite subsets of $\w$ such that
\begin{itemize}
\item for any $F_1,\dots,F_n\in\F$ the intersection $\bigcap_{i=1}^nF_i$ is infinite;
\item for any infinite subset $I\subseteq\w$ there exists a set $F\in\F$ such that $I\setminus F$ is infinite.
\end{itemize}
It is known  that $\w_1\le\mathfrak p\le\mathfrak b\le\mathfrak c$, and $\mathfrak p=\mathfrak c$ under Martin's Axiom (see, \cite{Blass}, \cite{vD}, \cite{Vau}).

We recall that the {\em character} $\chi(x;X)$ of a point $x$ in a topological space $X$ is the smallest cardinality of a neighborhood base at $x$. The cardinal $$\chi(X):=\sup_{x\in X}\chi(x;X)$$ is called the {\em character} of the space $X$. The definition of the cardinal $\mathfrak p$ implies the following known fact.

\begin{proposition}\label{p:p2} Let $A$ be a countable set in a topological space $X$ and $x\in\bar A$. If $\chi(x;X)<\mathfrak p$, then the set $A$ contains a sequence that converges to $x$.
\end{proposition}

Proposition~\ref{p:p2} and \ref{p:ss} imply the following corollary. 

\begin{corollary}\label{c:p1} Each separable compactification of character $<\mathfrak p$ is  soft.
\end{corollary}

Unifying Corollaries~\ref{c:main} and \ref{c:fu}, we obtain

\begin{corollary}\label{c:FU2} A compact Hausdorff space $K$ is homeomorphic to the Higson compactification of some finitary coarse space if $K$ is Fr\'echet-Urysohn and has dense set of isolated points.
\end{corollary}

Also, Corollaries~\ref{c:CHR}, \ref{c:CH0} and \ref{c:fu} imply the following CH-improvements of Corollary~\ref{c:FU2}.

\begin{corollary} Under $\as=\mathfrak c$, a separable compact Hausdorff space $K$ is homeomorphic to the Higson compactification of some cellular finitary coarse space if the space $K$ is Fr\'echet-Urysohn and has dense set of isolated points.
\end{corollary}

\begin{corollary} Under $\as=\mathfrak c$, a zero-dimensional separable compact Hausdorff space $K$ is homeomorphic to the binary compactification of some cellular finitary coarse space if the space $K$ is Fr\'echet-Urysohn and has dense set of isolated points.
\end{corollary}

\section{Examples of compactifications which are not soft}\label{s:nonsoft}

In this section we present some examples of compactifications which are not soft and hence are not homeomorphic to the Higson compactifications of finitary coarse spaces.

There are two sources of such examples. One source comes from compactifications that are composed with two overlapping  pieces with different topological properties.

\begin{proposition}\label{p:not1} Let $K$ be a compactification containing two disjoint sets $A,B\subseteq K'$ such that $\bar A\cap \bar B\ne\emptyset$ and for any infinite sets $I\subseteq A$ and $J\subseteq B$ with $\bar I\setminus I=\bar J\setminus J$ the spaces $\bar I$ and $\bar J$ are not homeomorphic. Then the compactification $K$ is not soft. If $A\cup B=K'$, then $K$ is not $\mathsf 2$-soft and hence is not homeomorphic to the Higson compactification of a finitary coarse space.
\end{proposition}

\begin{proof} Assuming that $K$ is soft, we can find a homeomorphism $h\in\Homeo(K,K')$ such that the set $I=\{x\in A:h(x)\in B\}\subseteq A$ is infinite and  so is the set $J=h(I)$. Take any accumulation point $x\in\bar I\setminus I$ and observe that $x=h(x)\in\bar J\setminus J$. So, $\bar I\setminus I\subseteq\bar J\setminus J$. By analogy we can prove that $\bar J\setminus J\subseteq \bar I\setminus I$. Since $h(\bar I)=\bar J$, the spaces $\bar I,\bar J$ are homeomorphic, but this contradicts our assumption.

If $A\cup B=X$, then we can consider the characteristic function $\chi_A:X\to\mathsf 2$. Assuming that $K$ is $\mathsf 2$-soft, we can find a homeomorphism $h\in\Homeo(K,X)$ such that the set $I=\{x\in A:h(x)\notin A\}$ is infinite and  so is the set $J=h(I)$. Then $\bar I\setminus I=\bar J\setminus J$ and the spaces  $\bar J\setminus J\subseteq \bar I\setminus I$. Since $h(\bar I)=\bar J$, the spaces $\bar I,\bar J$ are homeomorphic, which contradicts our assumption. This contradiction shows that the compactification $K$ is not $\mathsf 2$-soft. By Proposition~\ref{p:sss}, $K$ is not $\IR$-soft and by Corollary~\ref{c:main},  $K$  is not homeomorphic to the Higson compactification of a finitary coarse space. 
\end{proof}

Proposition~\ref{p:not1} allows us to construct a more concrete example of a non-soft compactification.

\begin{example} Let $K$ be a compactification containing two disjoint sets $A,B\subseteq K'$ such that $\bar A\cap \bar B\ne\emptyset$, $\bar A$ is homeomorphic to $\beta\w$ and $\bar B$ is first-countable.
Then the compactification $K$ is not soft. If $A\cup B=K'$, then $K$ is not $\mathsf 2$-soft and hence is not homeomorphic to the Higson compactification of a finitary coarse space.
\end{example}

The other source of non-soft compactifications are compactifications $K$ with small permutation group $G_{K',K}$. For a set $X$ by $FS_{X}$ we denote the group of permutations $g:X\to X$ that have finite support $\supp(g):=\{x\in X:g(x)\ne x\}$.
It is clear that $FS_{K'}\subseteq G_{K',K}$ for any compactification $K$.
The Stone-\v Cech compactification $\beta X$ of a discrete space $X$ has $G_{X,\beta X}=FS_X$.

\begin{proposition} A compactification $K$ with $G_{K',K}=FS_{K'}$ is equal to the Higson compactification of its permutation coarse space $(K',\E_{K',K})$ if and only if $K=\beta K'$.
\end{proposition}

\begin{proof} Let $X$ be the set of isolated points of the compactification $K$. By our assumption, $G_{X,K}=FS_X$. Then for any entourage $E\in E_{FS_X}$ the complement $E\setminus \Delta_X$ is finite, which implies that each bounded function $\varphi:K'\to\IR$ is slowly oscillating and then the Higson compactification $\bar X_{\E_{X,K}}$ of the coarse space $(X,\E_{X,K})$ is equivalent to  the Stone-\v Cech compactification of $X$. Now we see that $K=\beta X$ if and only if $\bar X_{\E_{X,K}}=\beta X$.
\end{proof}

We recall (see \cite{Blass}, \cite{vD}, \cite{Vau}) that $\mathfrak s$ is the smallest cardinality $|\F|$ of a family $\F$ of subsets of $\w$ such that for any infinite set $I\subseteq \w$ there exists a set $F\in\F$ such that both sets $I\cap F$ and $I\setminus F$ are infinite. By Theorem 6.1 in   \cite{vD}, each compact space of weight $<\mathfrak s$ is sequentially compact (so contains many non-trivial convergent sequences).

\begin{proposition} A compactification $K$ has $G_{K',K}\ne FS_{K'}$ if $K$ is infinite and one of the following conditions is satisfied:
\begin{enumerate}
\item $K$ contains a non-trivial convergent sequence of isolated points;
\item $K$ has weight $w(K)<\mathfrak s$;
\item $K$ is separable and has character $\chi(X)<\mathfrak p$.
\end{enumerate}
\end{proposition}

\begin{proof} 1. Assume that the compactification $K$ contains a convergent sequence $(x_n)_{n\in\w}$ consisting of pairwise distinct isolated points of $K$. It can be shown that the homeomorphism $h\in \mathcal H(K,K')$ defined by the formula
$$h(x)=\begin{cases}x_{2n+1}&\mbox{if $x=x_{2n}$ for some $n\in\w$};\\
x_{2n}&\mbox{if $x=x_{2n+1}$ for some $n\in\w$};\\
x&\mbox{otherwise};
\end{cases}
$$
determines the bijection $h{\restriction}K'\in G_{K',K}\setminus FS_{K'}$, which implies that $G_{K',K}\ne FS_{K'}$.
\smallskip

2. If $K$ has weight $w(K)<\mathfrak s$, then each sequence in $K$ has a convergent subsequence. Since $K$ is infinite and the set $K'$ of isolated points is dense in $K$, it contains a non-trivial convergent sequence. Now the preceding item applies.
\smallskip

3. If $K$ is separable and has $\chi(K)<\mathfrak p$, then the set $K'$ of isolated points is countable. Applying Proposition~\ref{p:p2}, we conclude that $K'$ contains a non-trivial sequence, convergent to some non-isolated point of $K$. The first item of this proposition implies that  $G_{K',K}\ne FS_{K'}$.
\end{proof}

By $\non(\M)$ we denote the smallest cardinality of a non-meager set in the Cantor cube $\{0,1\}^\w$.  It is known (see \cite{Blass}, \cite{Vau}) that the strict inequality $\non(\M)<\mathfrak c$ is consistent.

\begin{proposition}
 There exists a zero-dimensional compactification $c\IN$ of the discrete countable space $\IN$ such that $w(c\IN)\le\non(\mathcal M)$ and $G_{\IN,c\IN}=FS_{\IN}$. If $\non(\mathcal M)<\mathfrak c$, then the compactification $c\IN$ is not homeomorphic to the Higson compactification of a finitary coarse space.
\end{proposition}

\begin{proof} Endow the family $[\IN]^{\le\w}$ of all subsets of $\IN$ with the natural compact metrizable topology, making $[\IN]^{\le\w}$ homeomorphic to the Cantor cube $\{0,1\}^\w$. A base of this topology consists of the sets $[F,E]=\{A\subseteq\IN:A\cap E=F\}$ where $F\subseteq E$ are finite subsets of $\IN$.

 Take any non-meager set $\A\subseteq [\IN]^{\le\w}$ of cardinality $|\A|=\non(\mathcal M)$. Replacing $\A$ by a larger family of the same cardinality, we can assume that $\A$ contains all finite subsets of $\IN$. For any set $A\in\A$ consider its characteristic function $\chi_A:\IN\to\{0,1\}$ (i.e., the unique function such that $\chi^{-1}_A(1)=A$). The functions $\chi_A$, $A\in\A$, determine the function $\chi_\A:\IN\to\{0,1\}^{\A}$, $\chi_\A:x\mapsto(\chi_A(x))_{A\in\A}$. 

Let $c\IN$ denote the closure of the set $\chi_\A(\IN)$  in the cube $\{0,1\}^\A$.
 Since the family $\A$ contains all finite subsets of $\IN$, the function $\chi_\A$ is injective, so we can identify $\IN$ with its image $\chi_\A(\IN)$ in $c\IN$ and consider $c\IN$ as a compactification of $\IN$.
 
 It is clear that $c\IN$ has weight $w(c\IN)\le|\A|=\non(\M)$. We claim that $G_{\IN,c\IN}=FS_{\IN}$. To derive a contradiction, assume that the compact space $c\IN$ admits a homeomorphism $h\in\Homeo(c\IN,\IN)$ such that the set $\{x\in\IN: h(x)\ne x\}$ is infinite. 
 
In the power-set $[\IN]^{\le\w}$ consider the subfamily
$$\A_h:=\big\{A\subseteq \IN:|\{x\in A:h(x)\notin A\}|=\w\big\}.$$
We claim that $\A_{h}$ is a dense $G_\delta$-set in $[\IN]^{\le\w}$.
For every $n\in\IN$ denote by $[\IN]^n$ the family of $n$-element subsets of $\IN$. For any sets $I,J\in[\IN]^n$ consider the open family
$$\A_{I,J}:=\{A\subseteq\IN: I\subseteq A,\;J\subseteq \IN\setminus A,\;h(I)=J\}$$in $[\IN]^{\le\w}$ and observe that the union $\bigcup_{I,J\in[\IN]^n}\A_{I,J}$ is dense in $[\IN]^{\le\w}$. Then the set
$$\A_{h}=\bigcap_{n\in\IN} \bigcup_{I,J\in[\IN]^n}\A_{I,J}$$is a dense $G_\delta$ in $[\IN]^{\le\w}$. Since the family $\A$ is not meager in $[\IN]^{\le\w}$, there exists a set $A\in\A\cap \A_{h}$.  By the definition of the topology of the space $c\IN$, the function $\chi_A:\IN\to\{0,1\}$ admits a continuous extension $\bar\chi_A:c\IN\to\{0,1\}$.

The definition of the family $\A_h$ guarantees that the set $\Omega=\{x\in A:h(x)\notin A\}$ is infinite. By the compactness of $c\IN$, there exists a point $z\in\bar\Omega\setminus\Omega\subseteq c\IN\setminus\IN$.
It follows that $\bar\chi_A(z)\subseteq \bar\chi_A(\bar \Omega)\subseteq\overline{\chi_A(\Omega)}=\{1\}$ and $\bar\chi_A(h(z))\in \bar\chi_A(\overline{h(\Omega)})\subseteq\bar \chi_A(\IN\setminus A)\subseteq\overline{\chi_A(\IN\setminus A)}=\{0\}$. But this contradicts the equality $h(z)=z$ following from $h\in\Homeo(c\IN,\IN)$. This contradiction shows that $G_{\IN,c\IN}=FS_{\IN}$.

Now assuming that $\non(\M)<\mathfrak c$, we prove that the compactification $c\IN$ is  not homeomorphic to the Higson corona of a finitary coarse space. To derive a contradiction, assume that $c\IN$ is  homeomorphic to the Higson corona of a finitary coarse space. By Lemma~\ref{l:soft}, the compactification $c\IN$ is $\IR$-soft.
Since $w(c\IN)\le\non(\M)<\mathfrak c=w(\beta\IN)$, the compactification  $c\IN$ is not homeomorphic to $\beta\IN$. Then there exists a bounded function $\varphi:\IN\to\IR$ that has no continuous extension to $c\IN$ and hence $\osc_z(\varphi)>0$ for some $z\in c\IN$. Since the compactification $c\IN$ is $\IR$-soft, there exist a homeomorphism $g\in \mathcal H(c\IN,\IN)$ and $\e>0$ such that the set $\{x\in \IN:|\varphi(g(x))-\varphi(x)|>\e\}$ is infinite. The homeomorphism $g$ determines a permutation $g{\restriction}\in G_{\IN,c\IN}\setminus FS_\IN$, which contradicts the equality $G_{\IN,c\IN}=FS_\IN$.
\end{proof} 

\section{Constructing a finitary coarse space with a given Higson or binary corona}\label{s:Hcor}

In this section we apply Theorem~\ref{t:main2} for constructing a finitary coarse space with a given Higson or binary corona. All compact spaces considered in this section are Hausdorff.

\begin{definition}\label{d:p} A compact Hausdorff space $K$ is called a {\em corona} if $K$ is homeomorphic to the remainder $c\mathbb N\setminus \mathbb N$ of suitable compactification $c\mathbb N$ of the discrete space $\mathbb N$ of positive integers. If the compactification $c\IN$ can be chosen to be soft (or $M$-soft for some proper metric space $M$), then the corona $K$ is called a {\em soft} (resp. {\em $M$-soft\/}).
\end{definition}

\begin{remark} Coronas can be equivalently defined as  continuous images of the remainder $\beta\IN\setminus\IN$ of the Stone-\v Cech compactification of $\IN$. Each corona $K$ has weight $w(K)\le\mathfrak c$ and cardinality $|K|\le 2^{\mathfrak c}$. It is known that the class of coronas  includes all separable compact spaces, all perfectly normal compacta (see, \cite{Przym}), and all compact spaces of weight $\le \w_1$ (see, \cite{Par} or \cite[1.3.4]{vM}). By a recent result of Hart \cite{Hart},  under CH each compact space of weight $\le\w_1$ is a soft corona. On the other hand, Bell \cite{Bell} constructed a consistent example of a separable compact space which is not a corona, and Dow observed that under his  principle (NT) introduced in \cite{Dow}, the compact Hausdorff space $K=\w_1+1+\w_1^*$ is not a soft corona, see Example in \cite{Hart}. 
\end{remark} 

The above remark and Proposition~\ref{p:sss} imply that for any compact Hausdorff space $K$ and any proper metric space $M$ we have the implications:
$$
\xymatrix{
\mbox{$\IR$-soft corona}\ar@{=>}[r]&\mbox{$M$-soft corona}\ar@{=>}[r]&\mbox{$\mathsf 2$-soft corona}\ar@{=>}[d]\\
\mbox{soft corona}\ar@{=>}[u]&\mbox{weight $\le\w_1$}\ar@{=>}[r]\ar_{\;\;+CH}[l]&\mbox{corona}\ar@{=>}[r]&\mbox{weight $\le\mathfrak c$}
}
$$

Definition~\ref{d:p} and Theorem~\ref{t:main2} imply the following characterizations. 

\begin{corollary}\label{c:corP} A compact Hausdorff space $K$ is homeomorphic to the Higson corona of a finitary coarse space $(X,\E)$ with $|X|=\w$ if and only if $K$ is an $\IR$-soft corona.
\end{corollary}

\begin{corollary}\label{c:corP} A zero-dimensional compact Hausdorff space $K$ is homeomorphic to the binary corona of a finitary coarse space $(X,\E)$ with $|X|=\w$ if and only if $K$ is a $\mathsf 2$-soft corona.
\end{corollary}

These corollaries rise the problem of recognizing $\IR$-soft coronas among coronas. 
We recall that a topological space $X$ is {\em perfectly normal} if it is normal and each closed subset of $X$ is of type $G_\delta$. By \cite{Przym}, perfectly normal compact spaces are coronas.

\begin{theorem}\label{t:Par} A compact space $K$ is a soft corona if one of the following conditions is satisfied:
\begin{enumerate}
\item $K$ is a corona of character $\chi(K)<\mathfrak p$;
\item $K$ is perfectly normal;
\item $K$ has weight $w(K)\le\w_1<\mathfrak p$;
\item $K$ has weight $w(K)\le\w_1=\mathfrak c$.
\end{enumerate}
\end{theorem}

\begin{proof} 1. Assume that $K$ is a corona with $\chi(K)<\mathfrak p$. Then  $K$ can be identified with the remainder $c\IN\setminus\IN$ of some compactification $c\IN$ of $\IN$. Taking into account that $\chi(K)<\mathfrak p$ and  $c\IN\setminus K=\IN$ is countable, we conclude that $c\IN$ has (pseudo)character $<\mathfrak p$. By Corollary~\ref{c:p1}, the compactification $c\IN$ is soft. Consequently, $K$ is a soft corona.
\smallskip

2. If $K$ is perfectly normal, then $K$ is a corona according to \cite{Przym}. Being perfectly normal, the space $K$ is first-countable. By the first statement, the space $K$ is a soft corona.
\smallskip

3. If $K$ has weight $w(K)\le\w_1<\mathfrak p$, then $K$ is a corona by \cite{Par} or \cite[1.3.4]{vM} and has character $\chi(K)\le w(K)<\mathfrak p$. Applying the first statement, we conclude that $K$ is a soft corona.
\smallskip

4. If $w(K)\le\w_1=\mathfrak c$, then $K$ is a soft corona by a result of Hart  \cite{Hart}.
\end{proof} 

Corollary~\ref{c:corP} and Theorems~\ref{t:Par}, \ref{t:mainCH} imply the following corollaries.

\begin{corollary}\label{c:Higson} (Under $\as=\mathfrak c$) a compact space $K$ is homeomorphic to the Higson corona of a (cellular) finitary coarse space $(X,\E)$ with $|X|=\w$ if one of the following conditions is satisfied:
\begin{enumerate}
\item $K$ is a corona of character $\chi(K)<\mathfrak p$;
\item $K$ is perfectly normal;
\item $w(K)\le\w_1<\mathfrak p$;
\item $w(K)\le\w_1=\mathfrak c$.
\end{enumerate}
\end{corollary}

\begin{corollary}\label{c:Higson0} (Under $\as=\mathfrak c$) a zero-dimensional compact space $K$ is homeomorphic to the binary corona of a (cellular) finitary coarse space $(X,\E)$ with $|X|=\w$ if one of the following conditions is satisfied:
\begin{enumerate}
\item $K$ is a corona of character $\chi(K)<\mathfrak p$;
\item $K$ is perfectly normal;
\item $w(K)\le\w_1<\mathfrak p$;
\item $w(K)\le\w_1=\mathfrak c$.
\end{enumerate}
\end{corollary}

\begin{remark} Corollary~\ref{c:Higson} can be considered as a generalization of the result of Mine and Yamashita \cite{MY} who proved that each compact metric space is homeomorphic to the Higson corona of a suitable coarse space (namely, a totally bounded locally compact metric space carrying the $C_0$ coarse structure of Wright \cite{Wri}).
\end{remark}

With the help of Theorems~\ref{t:mainCH} and \ref{t:Par}(4), we can improve Corollary~\ref{c:Higson}(4) as follows.

\begin{corollary}\label{c:CH-H} Under CH, each compact space of weight $\le\w_1$ is homeomorphic to the Higson corona of some cellular finitary coarse space $(X,\E)$ with $|X|=\w$.
\end{corollary}

\begin{corollary}\label{c:CH-0} Under CH, each zero-dimensional compact space of weight $\le\w_1$ is homeomorphic to the binary corona of some cellular finitary coarse space $(X,\E)$ with $|X|=\w$.
\end{corollary}

\begin{remark} Corollary~\ref{c:CH-H} should be compared with Theorem 4 \cite{Prot_bc} (see also Theorem~11 in \cite{DKU}) saying that the Higson corona of any ordinal cellular coarse space is zero-dimensional.
\end{remark}

\begin{remark} We do not know whether the Continuum Hypothesis can be removed from Corollaries~\ref{c:CH-H} and \ref{c:CH-0}. 
\end{remark}

By \cite{Hart}, under CH, each compact space of weight $\le\w_1$ is a soft corona and hence an $\IR$-soft corona. We do not know if the latter fact remains true in ZFC.

\begin{problem} Is each compact Hausdorff space of weight $\le\w_1$ an $\IR$-soft corona? A $\mathsf 2$-soft corona?
\end{problem}

Another problem is motivated by Corollary~\ref{c:fu}.

\begin{problem} Is each (Fr\'echet-Urysohn) corona $\IR$-soft?
\end{problem}

By \cite{Hart}, the Dow Principle (NT) implies that the linearly ordered compact space $\w_1+1+\w_1^*$ is not a soft corona. Here $\w_1^*$ stands for the space $\w_1$ with the reverse linear order. On the other hand, the space $\w_1+1+\w_1^*$ has the following ZFC-property.

\begin{proposition}\label{p:sPar} The compact space $\w_1+1+\w_1^*$ is an $\IR$-soft corona.
\end{proposition}

\begin{proof} If $\w_1<\mathfrak p$, then the compact space $\w_1+1+\w_1^*$ is a  soft corona by Theorem~\ref{t:Par}(3) and Proposition~\ref{p:sss}. It remains to consider the case $\w_1=\mathfrak p$.
Write the linearly ordered space $L=\w_1+1+\w_1^*$ as $$L=\{x_\alpha\}_{\alpha\in\w_1}\cup\{z\}\cup\{y_\alpha\}_{\alpha\in\w_1}$$ where $$x_\alpha<x_\beta<z<y_\beta<x_\alpha$$for any $\alpha<\beta<\w_1$. For two elements $a\le b$ in $L$, consider the order intervals 
$$
\begin{aligned}
&L[a,b]:=\{x\in L:a\le x\le b\},\;L[a,b):=\{x\in L:a\le x<b\},\\
&L(a,b]:=\{x\in L:a<x\le b\},\;L(a,b):=\{x\in L:a<x<b\}
\end{aligned}
$$ with end-points $a,b$.

Choose any countable set $D$, which is disjoint with $L$. 
According to Theorem 1.2 in \cite{NV}, the equality $\w_1=\mathfrak p$ implies the existence of a tight Hausdorff gap on $D$, which is a transfinite sequence $\langle (X_\alpha,Y_\alpha)\rangle_{\alpha\in\w_1}$ of pairs of infinite subsets of $D$ satisfying the following conditions:
\begin{enumerate}  
\item $X_\alpha\subset^* X_\beta$ and $Y_\alpha\subset^* Y_\beta$ for any countable ordinals $\alpha<\beta$;
\item $X_\alpha\cap Y_\alpha$ is finite for every $\alpha\in\w_1$;
\item for any disjoint sets $X,Y\subseteq D$ there exists $\alpha\in\w_1$ such that $X_\alpha\not\subseteq^* X$ or $Y_\alpha\not\subseteq^* Y$;
\item for any infinite set $I\subseteq D$ there exists $\alpha\in\w_1$ such that $I\cap(X_\alpha\cup Y_\alpha)$ is infinite.
\end{enumerate} 
Here given sets $A,B$, we write $A\subseteq^* B$ (and $A\subset^* B$)  if the complement $A\setminus B$ is finite (and the complement $B\setminus A$ is infinite).

Consider the space $K=L\cup D$ endowed with the topology $\tau$ consisting of the sets $U\subseteq K$ satisfying the following conditions:
\begin{itemize}
\item if $x_\alpha\in U$ for some ordinal $\alpha\in\w_1$, then there exists an ordinal $\beta<\alpha$ such that $L(x_\beta,x_\alpha]\cup (X_\alpha\setminus X_\beta)\subseteq^* U$;
\item if $y_\alpha\in U$ for  some ordinal $\alpha\in\w_1$, then there exists an ordinal $\beta<\alpha$ such that $L[y_\alpha,y_\beta)\cup (Y_\alpha\setminus Y_\beta)\subseteq^* U$;
\item if $z\in U$, then there exists an ordinal $\alpha\in\w_1$ such that $L(x_\alpha,y_\alpha)\cup (D\setminus (X_\alpha\cup Y_\alpha))\subseteq^* U$.
\end{itemize}

\begin{claim} $K$ is a compact Hausdorff space.
\end{claim} 

\begin{proof} The Hausdorff property of $K$ follows from the definition of the topology $\tau$. Also the definition of the topology $\tau$ implies that the induced topology on $L$ coincides with the topology generated by the linear order, which implies that the subspace $L$ of $K$ is homeomorphic to the compact linearly ordered space $\w_1+1+\w_1^*$. By \cite[3.12.1]{Eng}, the compactness of $K$ will follows as soon as we check that each infinite subset $S\subseteq K$ has a complete accumulation point, i.e., a point $x\in K$ such that for any neighborhood $O_x\subseteq K$ we have $|O_x\cap S|=|S|$. If $S$ is uncountable, then $|S|=\w_1=|S\cap L|$ and $S\cap L$ has a complete accumulation point by the compactness of $L$. So, we assume that $S$ is countable. If $S\cap L$ is infinite, then $S\cap L$ has a complete accumulation point by the compactness of $L$. So, we assume that $S\cap L$ is finite and hence $S\cap D$ is infinite. Using the condition (5), we can find the smallest cardinal $\alpha$ such that $S\cap D\cap (X_\alpha\cup Y_\alpha)$ is infinite. Then $x_\alpha$ or $y_\alpha$ is a (complete) accumulation point of $S$. 
\end{proof}

The definition of the topology $\tau$ ensures that $D$ coincides with the set $K'$ of isolated points of the compact space $K$. 

\begin{claim}\label{cl:cont} Any continuous function $\varphi:K\setminus \{z\}\to\IR$ has a continuous extension $\bar\varphi:K\to\IR$.
\end{claim}

\begin{proof}  It suffices to prove that $\osc_z(\varphi)=0$. To derive a contradiction, assume that $\e:=\osc_z(\varphi)>0$. 

Since any closed unbounded sets in $[0,\w_1)$ have non-empty intersection, the continuous function $\varphi$ is eventually constant on the sets $L[x_0,z)$ and $L(z,y_0]$. This means that for some $\gamma\in\w_1$  and real numbers $a,b$ the sets $\varphi(L[x_\alpha,z))$ and $\varphi(L(z,y_\alpha])$ are equal to the singletons $\{a\}$ and $\{b\}$, respectively. By the normality and zero-dimensionality of the compact space $K$, the compact set $L[x_0,x_\gamma]$ has a clopen neighborhood $U_\gamma$ in $K$ which is disjoint with the compact set $L[x_{\gamma+1},y_0]$. By analogy, the compact set $L[y_\gamma,y_0]$ has a clopen neighborhood $W_\gamma$ in $K$ which is disjoint with the compact set $L[x_0,y_{\gamma+1}]$.

We claim that $a=b$. Assuming that $a\ne b$, put $\delta:=\frac12|a-b|$ and consider the open neighborhood $V=\{x\in K\setminus\{z\}:|\varphi(x)-a|<\delta\}$ of $L[x_\gamma,z)$ in $K$. 

 For every $\alpha\in\w_1$, the inclusion $L[x_0,x_\alpha]\subseteq (U_\gamma\setminus W_\gamma)\cup V$ implies $X_\alpha\subseteq^*(U_\gamma\setminus W_\gamma)\cup V$.
 
The condition (3) yields a countable ordinal $\alpha>\gamma$ such that the set $Y_\alpha\cap(U_\gamma\setminus W_\gamma)\cup V$ is infinite and hence has  an accumulation point $y$ in $K$. The point $y$ belongs to $L\cap \overline{Y_\alpha}\cap (U_\gamma\setminus W_\gamma)\subseteq L[y_\alpha,y_0]\setminus W_\gamma\subseteq L[y_\alpha,y_\gamma)$ and hence has $\bar\varphi(y)=b$. Taking into account that $y\in L[y_\alpha,y_\gamma)$ and $L[y_\alpha,y_\gamma)\cap U_\gamma=\emptyset$, we conclude that $y\in\overline{V}$ and hence $b=\varphi(y)\in\overline{\varphi(V)}\subseteq [a-\delta,a+\delta]$, which contradicts the choice of $\delta=\frac12|a-b|$. This contradiction shows that $a=b$. 

Now consider the open neighborhood $W=V_\gamma\cup W_\gamma\cup \{x\in K\setminus \{z\}:|\bar\varphi(x)-a|<\frac13\e\}$ of $L\setminus\{z\}$ in $K\setminus\{z\}$. Then the set $K\setminus W$ is closed and accumulates at $z$ as $\diam(\varphi(W\setminus(V_\gamma\cup W_\gamma)))\le\frac23\e<\e=\osc_z(\varphi)$. Moreover, the continuity of $\varphi$ ensures that $z$ is a unique accumulation point of the set $D\setminus W$. But this contradicts the condition (4). This contradiction shows that $\osc_z(\varphi)=0$ and hence $\varphi$ has a  continuous extension $\bar\varphi:K\to\IR$.
\end{proof}

\begin{claim}\label{cl:soft} The compactification $K$ is $\IR$-soft. 
\end{claim}

\begin{proof} Given any bounded function $\varphi:D\to \IR$ with non-empty set $\Omega=\{x\in K:\osc_x(\varphi)>0\}$, we should find a homeomorphism $h\in\mathcal H(K,D)$ and $\e>0$ such that the set $\{y\in D:|\varphi(h(y))-\varphi(y)|>\e\}$ is infinite. By Claim~\ref{cl:cont}, the set $\Omega$ contains some point $x\in L\setminus\{z\}$. Using the first countability of $K$ at $x$, we can repeat the argument of the proof of Proposition~\ref{p:ss} and find a homeomorphism $h\in\mathcal H(K,D)$ and $\e>0$ such that the set $\{y\in D:|\varphi(h(y))-\varphi(y)|>\e\}$ is infinite.
\end{proof}

Since the remainder $L$ of the compactification $K$ is homeomorphic to $\w_1+1+\w_1^*$, Claim~\ref{cl:soft} implies that the space $\w_1+1+\w_1^*$ is an $\IR$-soft corona. 
\end{proof}

 \section{Acknowledgements}
 
 The authors express their sincere thanks to KP Hart and Todd Eisworth for their answers to the MO-questions\footnote{\tt https://mathoverflow.net/q/309583, https://mathoverflow.net/q/351615} of the first author. Those answers were used in the proofs of Theorem~\ref{t:Par}(4) and Proposition~\ref{p:sPar}, respectively.


\end{document}